\documentclass[12pt]{amsart}
\usepackage{amssymb,amsmath,amsthm,mathrsfs,multirow,enumerate}

\usepackage{graphicx} 
\usepackage[all]{xy}
\usepackage{braket}

\numberwithin{equation}{section} 

\theoremstyle{plain}
\newtheorem{thm}{Theorem}[section]
\newtheorem{cor}[thm]{Corollary}
\newtheorem{lemma}[thm]{Lemma}

\theoremstyle{definition}

\newtheorem{ex}{Example}

\newtheorem*{rmk}{Remark}

\newcommand{\diag}{\mathop{\mathrm{diag}}}

\newcommand{\tr}{\mathrm{Tr}}

\newcommand{\bbm}{\begin{bmatrix}}
\newcommand{\ebm}{\end{bmatrix}}
\DeclareMathOperator{\Aut}{\rm Aut}

\DeclareMathOperator{\cha}{\rm char}

\begin{document}

\title[Orthogonal abelian Cartan subalgebra decomposition]{Orthogonal abelian Cartan subalgebra decomposition of $\mathfrak{sl}_n$ over a finite commutative ring}
\author{Songpon Sriwongsa and Yi Ming Zou}

\address{Songpon Sriwongsa\\ Department of Mathematical Sciences \\ University of Wisconsin-Milwaukee\\ USA}
\email{\tt songpon@uwm.edu}

\address{Yi Ming Zou\\ Department of Mathematical Sciences \\ University of Wisconsin-Milwaukee\\ USA}
\email{\tt ymzou@uwm.edu}



\maketitle

\begin{abstract}
Orthogonal decomposition of the special linear Lie algebra over the complex numbers was studied in the early 1980s and attracted further attentions in the past decade due to its application in quantum information theory. In this paper, we study this decomposition problem of the special linear Lie algebra over a finite commutative ring with identity.
\end{abstract}

\section{Introduction}\label{intro}
Let $\mathfrak{L}$ be a Lie algebra over $\mathbb{C}$. An {\it orthogonal decomposition} (OD) of $\mathfrak{L}$ is a decomposition of $\mathfrak{L}$ into a direct sum of Cartan subalgebras which are pairwise orthogonal with respect to the Killing form. Orthogonal decompositions of Lie algebras were studied as early as in \cite{T76} by Thompson and used for the construction of a special finite simple group. The theory of such decompositions of simple Lie algebras of types $A, B, C$ and $D$ over $\mathbb{C}$ was developed by Kostrikin and collaborators in the 1980s \cite{KK81, KK83, KT94}. The OD problem of $\mathfrak{sl}_n(\mathbb{C})$ is related to other fields such as mutually unbiased bases (MUBs) in $\mathbb{C}^n$ which have applications in information theory \cite{DE10,R09}. Boykin et. al.  established a connection between the problem of constructing maximal collections of MUBs and the existence problem of OD of $\mathfrak{sl}_n(\mathbb{C})$ \cite{BS07}. It was conjectured in \cite{KK81}, the so-called Winnie-the-Pooh conjecture, that the Lie algebra $\mathfrak{sl}_n(\mathbb{C})$ has an OD if and only if $n$ is a power of a prime integer. This would imply the nonexistence of $n+1$ MUBs in the $n$-dimensional complex space when $n$ is not a prime power \cite{BS07}. The only if part of the conjecture is still open. On the other hand, if $n$ is a composite number which is not a prime power, the maximum collection of pairwise orthogonal Cartan subalgebras of $\mathfrak{sl}_n(\mathbb{C})$ is unknown. This is the case even when $n$ is the first positive composite number $6$. For some more recent developments when $n = 6$, see \cite{BZ16}.

In this paper, we consider the OD problem of the special linear Lie algebra $\mathfrak{sl}_n$ over a finite commutative ring $R$ with identity. One of our motivations was to see if we can shed more light on the OD problem in the none prime power case such as the case of $n = 6$ by considering the problem over finite commutative rings, since detailed computations are possible for small commutative rings.
Lie algebras over $R$ are modular Lie algebras. These Lie algebras, in particular when $R$ is a field of positive characteristic, have arisen in other areas of mathematics. For more informations, we refer the reader to \cite{S67} and the references therein. 

 Let $\mathfrak{L}$ be a Lie algebra over $R$. Recall that a subalgebra $H$ of $\mathfrak{L}$ is a {\it Cartan subalgebra} if it is a nilpotent subalgebra which is its own normalizer. In contrast to Lie algebras over the complex numbers, where every Cartan subalgebra is abelian, in the modular case, not every Cartan subalgebra is abelian. Here, we consider the orthogonal decomposition of 
 \[
 \mathfrak{sl}_n(R) = \{ n \times n \text{ traceless matrices over } R \}
 \]
 into abelian Cartan subalgebras and use the abbreviation ODAC (AC for abelian Cartan). The orthogonality is defined via the Killing form:
$
 K(A, B) := Tr(adA \cdot adB)
 $
and  
\[
 K(A, B) = 2nTr (AB)
\]
if $A, B \in \mathfrak{sl}_n(R)$. Thus, an ODAC of $\mathfrak{sl}_n(R)$ is a decomposition
\[
\mathfrak{sl}_n(R) = H_0 \oplus H_1 \oplus \ldots \oplus H_k
\]
where the $H_i$'s are pairwise orthogonal abelian Cartan subalgebra of $\mathfrak{sl}_n(R)$.
\begin{ex}\label{two}
Assume that   $2 \nmid \cha(R) $. Then $ \mathfrak{sl}_2(R)$ has an ODAC
\[
\mathfrak{sl}_2(R) =\bigg \langle 
\begin{pmatrix}
1 & 0 \\
0 & -1
\end{pmatrix}
\bigg \rangle_R
\oplus
\bigg \langle 
\begin{pmatrix}
0 & 1 \\
-1 & 0
\end{pmatrix}
\bigg \rangle_R
\oplus
\bigg \langle 
\begin{pmatrix}
0 & 1 \\
1 & 0
\end{pmatrix}
\bigg \rangle_R.
\]
\end{ex}

In Section \ref{small}, we first consider a special type of ODAC (the so-called classical type) of the cases $n = 2, 3$ over a finite field. The observations in these special cases will then be used in Section \ref{main} to derive the main results for $n \geq 2$, these results provide sufficient conditions for the existence of an ODAC of $\mathfrak{sl}_n$ over a finite commutative ring with identity. In the cases of a finite local ring and a finite field, the verifications of these conditions are straight forward for the given ring and field since the needed information is readily obtained from their structures. In Section \ref{sum}, we conclude the paper with some comments.
                    
\section{ODAC for $n = 2, 3$ when $R$ is a finite field}\label{small}
 
Suppose that $n = 2, 3$ and $\mathbb{F}_q$ is a finite field of $q = p^m$ elements with characteristic $p \neq 2, 3$. We recall \cite{S67} that  a Lie algebra $\mathfrak{L}$ over $\mathbb{F}_q$ is {\it classical} if:
\begin{enumerate}[(i)]
	\item the center of $\mathfrak{L}$ is zero;
	\item $[\mathfrak{L}, \mathfrak{L}] = \mathfrak{L}$;
	\item $\mathfrak{L}$ has a abelian Cartan subalgebra $H$, relative to which:
	\begin{enumerate}[(a)]
		\item $\mathfrak{L} = \oplus \mathfrak{L}_\alpha$, where $[x, h] = \alpha(h)x$ for all $x \in \mathfrak{L}_\alpha, h \in H$;
		\item if $\alpha \neq 0$ is a root, $[\mathfrak{L}_\alpha, \mathfrak{L}_{-\alpha}]$ is one-dimensional;
		\item if $\alpha$ and $\beta$ are roots, and if $\beta \neq 0$, then not all $\alpha + k\beta$ are roots, where $1 \leq k \leq p - 1$.
	\end{enumerate}
Such an $H$ satisfying (a), (b) and (c) is called a {\it classical} Cartan subalgebra.
\end{enumerate}
An ODAC is said to be {\it classical} if all of its components are classical. 

\begin{ex}\label{twoZ7}
From Example \ref{two},	
 $ \mathfrak{sl}_2(\mathbb{Z}_7)$ has an ODAC
	\[
	\mathfrak{sl}_2(R) =\bigg \langle 
	\begin{pmatrix}
	1 & 0 \\
	0 & -1
	\end{pmatrix}
	\bigg \rangle_{\mathbb{Z}_7}
	\oplus
	\bigg \langle 
	\begin{pmatrix}
	0 & 1 \\
	-1 & 0
	\end{pmatrix}
	\bigg \rangle_{\mathbb{Z}_7}
	\oplus
	\bigg \langle 
	\begin{pmatrix}
	0 & 1 \\
	1 & 0
	\end{pmatrix}
	\bigg \rangle_{\mathbb{Z}_7}.
	\]
	However, it is not classical because $\sqrt{-1}$ is undefined here and so the adjoint action of the second matrix is not semisimple, i.e., $\mathfrak{sl}_2(R)$ does not have a root subspace decomposition relative to the second summand.
\end{ex}

In the finite field case, we can consider classical ODAC of the Lie algebra $\mathfrak{sl}_n$. Assume that $\mathfrak{sl}_n(\mathbb{F}_q)$ is classical, then $\cha(\mathbb{F}_q)$ is not equal to $n$ (otherwise the identity matrix would be in $\mathfrak{sl}_n(\mathbb{F}_q)$ and so the center would be nontrivial), and all its classical Cartan subalgebras are conjugate \cite{S67}. Let $H_0$ be the classical Cartan subalgebra of $\mathfrak{sl}_n{(\mathbb{F}_q)}$ consisting of the diagonal matrices. Note that the conjugation preserves the orthogonality with respect to the Killing form $K$. If $\mathfrak{sl}_n(\mathbb{F}_q)$ has a classical ODAC, then we can assume that one of the components is $H_0$. Let $H$ be a classical Cartan subalgebra of $\mathfrak{sl}_n{(\mathbb{F}_q)}$ orthogonal to $H_0$ with respect to $K$. According to the corollary to Lemma II.1.2 of \cite{S67}, $K$ is nondegenerate, and so is its restriction to $H_0$.
Since $H$ and $H_0$ are conjugate, $K_{|_H}$ is also non-degenerate.

We have the following lemma.

\begin{lemma}\label{H}
	Under the above setting, we have the following.
	\begin{enumerate}[(1)]
		\item If $n = 2$, then 
		\[
		H = \bigg \langle 
		\begin{pmatrix}
		0 & 1 \\
		a & 0
		\end{pmatrix}
		\bigg \rangle_{\mathbb{F}_q} \]
		for some $a \neq 0$.
		\item If $n = 3$, then 
		\[
		H = \Bigg \langle 
		\begin{pmatrix}
		0 & 1 & 0\\
		0 & 0 & a \\
		ab & 0 & 0
		\end{pmatrix},
		\begin{pmatrix}
		0 & 0 & 1\\
		ab & 0 & 0 \\
		0 & b & 0
		\end{pmatrix}
		\Bigg \rangle_{\mathbb{F}_q} \]
		for some $a, b \neq 0$.
	\end{enumerate}	
\end{lemma}
\begin{proof}
We will give the proof for $n = 3$ since similar arguments apply to the case $n = 2$.
We first prove the following assertions:
	\begin{enumerate}[(a)]
		\item Every matrix in $H$ has a zero  diagonal.
		\item Every nonzero matrix in $H$ has no zero row nor zero column.
		\item $H$ admits a basis $\{ A_2, A_3 \}$ satisfying the conditions below, $k = 2, 3$:
		\begin{enumerate}[(i)]
			\item The first row of $A_k$ has $1$ in the $k$-th position and $0$ elsewhere.
			\item The first position of the $k$-th column of $A_k$ is the only nonzero element in that column.
			\item The $j$-th row of $A_k$ coincides with the $k$-th row of $A_j$.
		\end{enumerate}
	\end{enumerate}
	First note that (a) holds since $H$ is orthogonal to $H_0$ and the characteristic of the field is not equal to $2$ or $3$. If we assume (b), then it follows that $H$ has a basis $\{A_2, A_3\}$ with property (i). We use the commutativity of $H$ to prove (iii). By (i), the $j$-th row of $A_k$ equals the first row of the product $A_jA_k$, but $A_jA_k = A_kA_j$, so it equals the $k$-th row of $A_j$. To prove (ii), we note that for $j \ge 2$, the $j$-th element of the $k$-th column of $A_k$ is the $k$-th element of the $j$-th row of $A_k$, so by (iii), it is equal to the $k$-th element of the $k$-th row of $A_j$, and therefore it is zero by (a). To prove (b), we assume the contrary. Without loss of generality, we may assume that there exists a nonzero matrix $A \in H$ whose first row is zero, i.e. 
	\[A = 
	\begin{pmatrix}
	0 & 0 & 0 \\
	a & 0 & b \\
	c & d & 0
	\end{pmatrix},
	\] 
	where $a, b, c, d$ are not all zero.
 	Note that $H$ has dimension two. Let $B$ be a nonzero matrix such that $H = \big \langle A, B \big \rangle_{\mathbb{F}_q}$. Write
 	\[
 	B =
 	\begin{pmatrix}
 	0 & x & y \\
 	u & 0 & z \\
 	v & w & 0
 	\end{pmatrix},
 	\] 
 	where $x, y, z, u, v,w$ are not all zero, then
 	\begin{align}\label{abe}
 	[A, B] = \begin{pmatrix}
 	-ax-cy  & -dy & -bx \\
 	bv - cz & bw + ax -dz & ay \\
 	du - aw & cx & -bw + cy + dz
 	\end{pmatrix}
 	\end{align}
 	equals zero because $H$ is abelian. If the product  $abcd$ is zero, then it can be verified that the Killing form would be degenerate on $H$. This can be done either by considering different cases or using computational algebra packages. Using the latter method, it is straightforward that the determinant of the Killing form is contained in the ideal $J \subset \mathbb{Z}[a, b, c, d, x, y, z, u, v, w]$ generated by the entries of $[A, B]$ and $abcd$. Codes in both Sage and Magma are provided in the Appendix for this purpose.
    Therefore, all $a, b, c$ and $d$ are nonzero. Now, by (\ref{abe}), $x = y = 0$ and we may assume that $a = 1$. So, we have $bv = cz, bw = dz$ and $du = w$. These can be reduced to $z = b u$ and $v = c u$. Since $B \neq 0$, $u \neq 0$. Again, we may assume that $u = 1$. Then $d = w, b = z$ and $c = v$, i.e., $A = B$, which contradicts the choice of $B$.
 	Therefore, (b) holds. 
 	
 	From the above discussions, $H$ admits a basis of the form 
	\[
	\Bigg\{ \begin{pmatrix}
	0 & 1 & 0\\
	x & 0 & a \\
	* & 0 & 0
	\end{pmatrix},
	\begin{pmatrix}
	0 & 0 & 1\\
	* & 0 & 0 \\
	y & b & 0
	\end{pmatrix} \Bigg\}
	\]
	 where $a, b \neq 0$. Since $H$ is abelian, $x = y = 0$ and $* = ab$.
\end{proof}

The above lemma leads us to the existence of an ODAC of $\mathfrak{sl}_n(\mathbb{F}_q)$, when $n = 2, 3$. For $n = 2$, the decomposition (some cases are classical) always exists because $\cha(\mathbb{F}_q) \neq 2$ (see Example \ref{two}). For $n = 3$, we state the result as a theorem. 

\begin{thm}\label{sl3}
	Let $\mathbb{F}_q$ be a finite field of $q = p^m$ elements with characteristic $p \neq 2, 3$. Then $\mathfrak{sl}_3(\mathbb{F}_q)$ has a classical ODAC if and only if $3 | (q - 1)$.
	In that case, for any primitive cube root of unity $u \in \mathbb{F}_q$, we have the following classical ODAC:
	\[
	\mathfrak{sl}_3(\mathbb{F}_q) = H_0 \oplus H_1 \oplus H_2 \oplus H_3,
	 \]
	where
	\begin{align*}
	 H_1 &=
	\Bigg \langle 
	\begin{pmatrix}
	0 & 1 & 0 \\
	0 & 0 & 1 \\
	1 & 0 & 0
	\end{pmatrix},
	\begin{pmatrix}
	0 & 0 & 1 \\
	1 & 0 & 0 \\
	0 & 1 & 0
	\end{pmatrix}
	\Bigg \rangle_{\mathbb{F}_q},\\
	H_2 &= 
	\Bigg \langle 
	\begin{pmatrix}
	0 & 1 & 0 \\
	0 & 0 & u \\
	u^2 & 0 & 0
	\end{pmatrix},
	\begin{pmatrix}
	0 & 0 & 1 \\
	u^2 & 0 & 0 \\
	0 & u & 0
	\end{pmatrix}
	\Bigg \rangle_{\mathbb{F}_q}, \\
		H_3 &= 
	\Bigg \langle 
	\begin{pmatrix}
	0 & 1 & 0 \\
	0 & 0 & u^2 \\
	u & 0 & 0
	\end{pmatrix},
	\begin{pmatrix}
	0 & 0 & 1 \\
	u & 0 & 0 \\
	0 & u^2 & 0
	\end{pmatrix}
	\Bigg \rangle_{\mathbb{F}_q}.
	\end{align*}

	\end{thm}
\begin{proof}
Assume that $3 | (q - 1)$. Since the unit group $\mathbb{F}_q^\times$ is cyclic and $|\mathbb{F}_q^\times| = q - 1$, there exists a primitive cube root of unity $u \in \mathbb{F}_q$.	
The verification that the given decomposition is an ODAC of $\mathfrak{sl}_3(\mathbb{F}_q)$ is straightforward. Let
\[
X = 
\begin{bmatrix}
1 & 1 & u \\
u & 1 & 1 \\
1 & u & 1
\end{bmatrix}
\text{ \ and \ }
Y = 
\begin{bmatrix}
u & u & 1 \\
u & 1 & u \\
u & u^2 & u^2
\end{bmatrix}.
\]
Then both $X$ and $Y$ are nonsingular. Note that the conjugation by $X$ (resp. $X^2$), changes $H_2$ (resp. $H_3$) to $H_0$, and the conjugation by $Y$ changes
$H_1$  to $H_0$. Since $H_0$ is a classical Cartan subalgebra, so are $H_1, H_2$ and $H_3$. Thus, this decomposition is classical. 

Conversely, suppose that $3 \nmid (q - 1) $ but $\mathfrak{sl}_3(\mathbb{F}_q)$ possesses a classical ODAC. Note that the decomposition of $\mathfrak{sl}_3(\mathbb{F}_q)$ has $4$ components. Then, up to conjugacy, we can assume that $H_0$ is one of the components and, by Lemma \ref{H}, all other components are of the forms 
\begin{align*}
H'_1 &= \Bigg \langle 
\begin{pmatrix}
0 & 1 & 0\\
0 & 0 & a \\
ab & 0 & 0
\end{pmatrix},
\begin{pmatrix}
0 & 0 & 1\\
ab & 0 & 0 \\
0 & b & 0
\end{pmatrix}
\Bigg \rangle_{\mathbb{F}_q} \\
H'_2 &= \Bigg \langle 
\begin{pmatrix}
0 & 1 & 0\\
0 & 0 & c \\
cd & 0 & 0
\end{pmatrix},
\begin{pmatrix}
0 & 0 & 1\\
cd & 0 & 0 \\
0 & d & 0
\end{pmatrix}
\Bigg \rangle_{\mathbb{F}_q} \\
H'_3 &= \Bigg \langle 
\begin{pmatrix}
0 & 1 & 0\\
0 & 0 & e \\
ef & 0 & 0
\end{pmatrix},
\begin{pmatrix}
0 & 0 & 1\\
ef & 0 & 0 \\
0 & f & 0
\end{pmatrix}
\Bigg \rangle_{\mathbb{F}_q}
\end{align*}
for some $a, b, c, d, e, f \neq 0$. By the orthogonality between $H'_1$ and $H'_2$, we have 
$
cd + ad + ab = 0 \text{ and } cd + cb + ab = 0.
$
Then
$d = a^{-1}cb$. Substituting $d$ in the first equation, we get
$
c^2 + a c + a^2 = 0.
$
However, since  $3 \nmid (q - 1)$, there is no primitive cube root of unity in $\mathbb{F}_q$, so the polynomial $x^2 + a x + a^2$ has no root in $\mathbb{F}_q$. This is a contradiction.
\end{proof}

\begin{rmk}
By the above theorem, if $\mathbb{F}_q$ does not have a primitive cube root of unity, then the number of pairwise orthogonal Cartan subalgebras in $\mathfrak{sl}_3(\mathbb{F}_q)$ is at most two. If $H_0$ and $H'_1$ is such a pair, then they must have the forms described in the theorem, and by \cite{S67}, $H_0$ and $H'_1$ are conjugate. However, the two matrices listed in $H'_1$ listed are not diagonalizable over $\mathbb{F}_q$, so there is no orthogonal pair of classical Cartan subalgebras in $\mathfrak{sl}_3(\mathbb{F}_q)$ in this case.
\end{rmk}

\section{Main results}\label{main}

We remark that every matrix described in Theorem \ref{sl3} is a product of a diagonal matrix and a permutation matrix. Let $u$ be a primitive cube root of unity and let

\begin{align*}
D =
\begin{pmatrix}
1 & 0 & 0 \\
0 & u & 0 \\
0 & 0 & u^2
\end{pmatrix}
\ \text{and} \
P =
\begin{pmatrix}
0 & 0 & 1 \\
1 & 0 & 0 \\
0 & 1 & 0
\end{pmatrix},
\end{align*}
then each matrix in Theorem \ref{sl3} is of the form $D^a P^b$ for some $a, b \in \{0, 1, 2\}$. We show that an ODAC of $\mathfrak{sl}_n(R)$ can be constructed under assumptions similar to the $n = 3$ case using the $n \times n$ version of matrices $D$ and $P$. 

The matrices $D$ and $P$ play a key role in the construction of OD for $\mathfrak{sl}_n(\mathbb{C})$ when $n = p^m$ for a prime integer $p$ and a positive integer $m$ \cite{KK81}. To use them in our construction here, some of the differences must be noted. The matrix $D$ requires the existence of a primitive $p$th root $u$ of unity, which always exists in the complex number case. But for a general finite commutative ring, the existence of $u$ needs to be assumed. Moreover, $u^{p - 1} + \ldots + u + 1 = 0$ holds in $\mathbb{C}$, but this may not hold in a general finite commutative ring unless $u - 1$ is a unit. In addition, one can use Lie's theorem to verify that the constructed decomposition is an OD in the complex number case \cite{KK81}, but Lie's theorem is not available in the general cases considered here.

\begin{thm}\label{commu}
	Let $R$ be a finite commutative ring with $1$. For a prime power $n = p^m$, if there exists a primitive $p$th root of unity $u \in R^\times$ such that $u - 1 \in R^\times$, then $\mathfrak{sl}_n(R)$ has an ODAC 
	\[
	\mathfrak{sl}_n(R) = H_\infty \oplus H_0 \oplus H_1 \oplus \ldots \oplus H_{n - 1}.
	\]
\end{thm}
\begin{proof} We first consider the case $m = 1$. For $n = 2$, see Example \ref{two}. Assume that $p > 2$.   
Let \[
D = \diag(1, u, \ldots, u^{p-1}) \text{ \ and \ } 
P =
\begin{pmatrix}
0 & 0 & \cdots & 0 & 1 \\
1 & 0 & \cdots & 0 & 0 \\
0 & 1 & \cdots & 0 & 0 \\
\vdots & \vdots & \ddots & \vdots & \vdots \\
0 & 0 & \cdots & 1 & 0
\end{pmatrix}.
\] 
Since $u^p - 1 = 0$ and $u - 1 \in R^\times$, $\tr(D) = 1 + u + u^2 + \ldots + u^{p - 1 } = 0$. Thus, $D$ and $P$ are matrices in $\mathfrak{sl}_p(R)$ and $p$ is the smallest positive integer such that $D^p = P^p = I$. For any $a, b \in \mathbb{Z}_p$, let $J_{(a, b)} = D^a P^b$. We have
\begin{align}\label{trzero}
 \tr J_{(a, b)} = 0 \Leftrightarrow (a, b) \neq (0, 0)
 \end{align}
 and
\begin{align}\label{commuPD}
P^bD^a = u^{-ab}D^aP^b.
\end{align}
The last equation implies 
\begin{align}\label{brac}
J_{(a, b)}J_{(c, d)} &= u^{-bc}J_{(a + c, b + d)} \text{ \ and \ } \\ \label{brack}
[J_{(a, b)}, J_{(c, d)}] &= (u^{-bc} - u^{- ad})J_{(a + c, b + d)}
\end{align}
for $a, b, c, d \in \mathbb{Z}_p$. For $a, k \in \mathbb{Z}_p$ with $a \neq 0$, $J_{(a, ka)}$ and $J_{(0, a)}$ are elements of $ \mathfrak{sl}_p(R)$ by (\ref{trzero}). For a fixed $k \in \mathbb{Z}_p$, it follows immediately from the definitions of $D$ and $P$ that $J_{(1,k)}, J_{(2,2k)}, \ldots, J_{(p - 1,k(p - 1))}$ are linearly independent. Construct the following free $R$-submodules:
\begin{align*}
H_k &= \langle J_{(a, ka)} | a \in \mathbb{Z}_p^\times \rangle_R, k \in \mathbb{Z}_p \text{ \ and } \\
 H_\infty &= \langle J_{(0, a)} | a \in \mathbb{Z}_p^\times \rangle_R = \langle P, P^2, \ldots, P^{p - 1} \rangle_R.
 \end{align*}
 By  (\ref{brack}), $H_\infty$ and $H_k$ are Lie subalgebras of $\mathfrak{sl}_p(R)$. 

Let 
\[
X = 
\begin{bmatrix}
1 & u^{\frac{p(p - 1)}{2}} & u^{\frac{(p - 1)(p - 2)}{2}}& \cdots & u^3 & u \\
u & 1 &  u^{\frac{p(p - 1)}{2}} & \cdots & u^6 & u^3 \\
u^3 & u & 1 & \cdots & u^{10} & u^6 \\
\vdots & \vdots  & \vdots &   \ddots & \vdots & \vdots \\
u^{\frac{(p - 1)(p - 2)}{2}} & u^{\frac{(p - 2)(p - 3)}{2}} & u^{\frac{(p - 3)(p - 4)}{2}} & \cdots & 1 & u^{\frac{p(p - 1)}{2}}\\
u^{\frac{p(p - 1)}{2}} &  u^{\frac{(p - 1)(p - 2)}{2}} & u^{\frac{(p - 2)(p - 3)}{2}} & \cdots & u & 1
\end{bmatrix}.
\] 
Since $p > 2$ and $1 - u$ is a unit, $X$ is invertible over $R$. It is straightforward to verify that $X^{-1}DPX = D$ and $X^{-1} P X = P$. Thus by (\ref{commuPD}), conjugation by the matrix $X$ shifts $H_0, H_1, \ldots, H_{p-1}$ cyclically and fixes $H_\infty$.
We show that 
\begin{align}\label{direct}
\mathfrak{sl}_p(R) = H_\infty \oplus H_0 \oplus H_1 \oplus \ldots \oplus H_{p - 1}.
\end{align}
It is clear from the construction that $H_0 \cap \sum_{j \neq 0} H_j = \{ 0 \}$. In particular, the sum is direct for $H_0$ and $H_\infty$. Thus, the sums for all $H_i$'s are also direct, and we have $H_\infty \oplus H_0 \oplus H_1 \oplus \ldots \oplus H_{p - 1}$, which is a free $R-$submodule of $\mathfrak{sl}_p(R)$. 
 But we also have 
\[
| \mathfrak{sl}_p(R) | = | H_\infty \oplus H_0 \oplus H_1 \oplus \ldots \oplus H_{p - 1} |. 
\]
Therefore, the equality (\ref{direct}) holds.

We prove that the decomposition (\ref{direct}) is pairwise orthogonal with respect to the Killing form $K(A, B) = 2p \tr(AB)$.  It is obvious that $H_\infty$ is orthogonal to all others $H_i$'s. Let $a, b \in \mathbb{Z}_p^\times, k_1, k_2 \in \mathbb{Z}_p$ with $k_1 \neq k_2$. Then 
$(a + b, k_1a + k_2b) \neq (0, 0)$
and so by (\ref{brac}), 
\begin{align*}
K(J_{(a, k_1a)}, J_{(b, k_2b)}) &= 2p \tr (J_{(a, k_1a)} J_{(b, k_2b)}) \\
& = 2p u^{-k_1ab} \tr(J_{(a + b, k_1a + k_2 b)}) \\
& = 0.
\end{align*}
Thus, $H_i$ and $H_j$ are orthogonal for all $i, j \in \mathbb{Z}_p$ and $i \neq j$. 

We now show that $H_k, (k \in \mathbb{Z}_p)$ and $H_\infty$ are abelian Cartan subalgebras. It is clear from the construction that both $H_0$ and $H_\infty$ are abelian. Moreover, $H_0$ is a Cartan subalgebra. Since $H_0, H_1, \ldots, H_{p-1}$ are conjugate, they are all abelian Cartan subalgebras. It remains to verify that $H_\infty$ is self normalizing. 
Recall that for all $k \in \mathbb{Z}_p$ and $a, b \in \mathbb{Z}_p^\times,
[J_{(a, k a)}, J_{(0,  b)}] = (1 - u^{-ab})J_{(a, ka + b)}$ is in $H_c$ for some $c \in \mathbb{Z}_p$. 
Now, let $ A \in N_{\mathfrak{sl}_p(R)}(H_\infty)$. Then  by (\ref{direct}), we can write
\[
A = \sum_{c = 1}^{p - 1} \Big(\sum_{j = 0}^{p - 1}(\alpha_{(c, j)}J_{(c, jc)}) + \beta_c J_{(0, c)} \Big),
\]
where $\alpha_{(c, j)}, \beta_c \in R$. For any basis element $J_{(0, a)}$ of $H_\infty$, we have
\begin{align*}
[A, J_{( 0, a)}] = \sum_{c = 1}^p \Big(\sum_{j = 0}^p(\alpha_{(c, j)}[J_{(c, jc)}, J_{(0, a)} ]) + \beta_c [J_{(0, c)}, J_{(0 ,a)}] \Big) \in H_\infty.
\end{align*}
This implies 
\[
\sum_{c = 1}^{p - 1} \sum_{j = 0}^{p - 1}(\alpha_{(c, j)} (1 - u^{-ac})J_{(c, jc + a)}) =
\sum_{c = 1}^{p - 1} \sum_{j = 0}^{p - 1}(\alpha_{(c, j)}[J_{(c, jc)}, J_{(0, a)} ]) \in H_\infty.
\]
This summation is also in $ \oplus_{i = 0}^{p - 1}H_i$.
Then by (\ref{direct}), it must be zero. For any $c \in \mathbb{Z}_p^\times, j \in \mathbb{Z}_p$, we can choose $a = -c^{-1}$ so the scalar $1 - u^{-ac} = 1 - u$ is a unit in $R$. So, $\alpha_{(c, j)} = 0$. Hence, $H_\infty = N_{\mathfrak{sl}_p(R)}(H_\infty)$. This completes the proof for the case $m = 1$.

Next suppose that $m \geq 2$. Let $W = \mathbb{F}_{p^m} \oplus \mathbb{F}_{p^m}$ be a $2m$-dimensional vector space over $\mathbb{F}_{p}$ equipped with a symplectic form $ \langle \cdot , \cdot \rangle : W \times W \rightarrow \mathbb{F}_p$ defined by the field trace as follows: for any elements $\vec{w} = (\alpha; \beta), \vec{w}' = (\alpha'; \beta') \in W$,
\[
\langle \vec{w}, \vec{w}'  \rangle = Tr _{\mathbb{F}_{p^m}/\mathbb{F}_{p}}(\alpha \beta' - \alpha' \beta).
\]
Then, by Corollary 3.3 of \cite{ZW02}, $W$ possesses a symplectic basis $\mathcal{B} = \{ \vec{e}_1, \ldots, \vec{e}_m, \vec{f}_1,\ldots, \vec{f}_m \}$ where $\{\vec{e}_1, \ldots, \vec{e}_m \}$ and $\{\vec{f}_1,\ldots, \vec{f}_m \}$ span the first and the second factor, respectively, such that

\[
\langle \vec{w}, \vec{w}' \rangle = \sum_{i = 1}^{m}(a_ib'_i - a'_ib_i),
\]
where $\vec{w} = \sum_{i = 1}^{m}(a_i\vec{e}_i + b_i\vec{f}_i)$ and $\vec{w}' = \sum_{i = 1}^{m}(a'_i\vec{e}_i + b'_i\vec{f}_i)$. With the basis $\mathcal{B}$, write each vector $\vec{w} \in W$ as
\[
\vec{w} = (a_1, \ldots, a_m; b_1, \ldots, b_m),
\]
and associate it with a matrix
\[
\mathcal{J}_{\vec{w}} = J_{(a_1, b_1)} \otimes J_{(a_2, b_2)} \otimes \cdots \otimes J_{(a_m, b_m)},
\]
where $\otimes$ denotes the Kronecker product of matrices, and $J_{(a_i, b_i)}$ is given as in the case $m = 1$ with a given primitive $p$th root of unity $u \in R^\times$ such that $u - 1 \in R^\times$ for all $i = 1, 2, \ldots, m$. Then the set $\{ \mathcal{J}_{\vec{w}} : 0 \neq \vec{w} \in W \}$ forms a basis of $\mathfrak{sl}_{p^m}(R)$ as a free $R$-module of rank $p^m + 1$. By the properties of Kronecker product, we have the following:
\begin{align}\label{brac1}
\mathcal{J}_{\vec{w}}\mathcal{J}_{\vec{w}'} &= u^{-\mathfrak{B}(\vec{w}, \vec{w}')}\mathcal{J}_{\vec{w} + \vec{w}'} \text{ \ and \ } \\ \label{brac1.5}
 [\mathcal{J}_{\vec{w}}, \mathcal{J}_{\vec{w}'}] &= (u^{-\mathfrak{B}(\vec{w}, \vec{w}')} -u^{-\mathfrak{B}(\vec{w}', \vec{w})})\mathcal{J}_{\vec{w} + \vec{w}'} \\
 &= u^{-\mathfrak{B}(\vec{w}', \vec{w})}(u^{\langle \vec{w}, \vec{w}' \rangle} - 1) \mathcal{J}_{\vec{w} + \vec{w}'}, \nonumber
\end{align} 
where 
\[
\mathfrak{B}(\vec{w}, \vec{w}') = \sum_{i = 1}^m a'_i b_i
\]
for all $\vec{w} = (a_1, \ldots, a_m; b_1, \ldots, b_m), \vec{w}' = (a'_1, \ldots, a'_m; b'_1, \ldots, b'_m) \in W$. 

Write $\vec{w} = (\alpha; \beta) \in W$, where $\alpha = (a_1, a_2, \ldots, a_m)$ and $\beta = (b_1, b_2, \ldots, b_m)$. Define 
\begin{align*}
 H_\infty =  \langle \mathcal{J}_{( 0; \lambda )} | \lambda \in \mathbb{F}^\times_{p^m} \rangle_R \text{ \ and \ } 
H_\alpha = \langle \mathcal{J}_{( \lambda; \alpha \lambda )} | \lambda \in \mathbb{F}^\times_{p^m} \rangle_R, 
\end{align*}
 where $\alpha \in \mathbb{F}_{p^m}$. Since the $\mathcal{J}_{\vec{w}}$'s are basis elements, we have  
 \begin{align}\label{direct2}
 \mathfrak{sl}_{p^m}(R) = H_\infty \oplus (\oplus_{\alpha \in \mathbb{F}_{p^m}} H_\alpha)
 \end{align}
 
 We show that all component $H_i$'s are pairwise orthogonal abelian Cartan subalgebras. It is clear that $\langle (\lambda; \alpha \lambda ), (\lambda';  \alpha \lambda') \rangle = \braket{(0; \lambda), (0; \lambda')} = 0$, so by (\ref{brac1.5}), all $H_\alpha$ and $H_\infty$ are abelian. To see that they are pairwise orthogonal, note that if $(\gamma; \delta) \neq (-\alpha; -\beta)$, then $\tr (\mathcal{J}_{(\alpha; \beta)}\mathcal{J}_{(\gamma; \delta)}) = 0$. Indeed, if $\lambda = (a_1, \ldots, a_m), \beta = (b_1, \ldots, b_m), \gamma = (a'_1, \ldots, a'_m), \delta = (b'_1, \ldots, b'_m)$ and $a_i \neq -a'_i$ for some $i \in \{1, \ldots, m\}$, then $a_i + a'_i \neq 0$ and $\tr J_{(a_i + a'_i, b_i + b'_i)} = 0$ (as in the case $m = 1$). By (\ref{brac1}) and the trace property of Kronecker product,
 \begin{align*} 
 \tr (\mathcal{J}_{(\alpha; \beta)}\mathcal{J}_{(\gamma; \delta)}) &= u^{-\mathfrak{B}((\alpha;\beta), (\gamma; \delta))} \tr (\mathcal{J}_{(a_1+a_1', \ldots, a_m + a_m' ; b_1 + b_1', \ldots, b_m + b_m')}) \\
  &= u^{-\mathfrak{B}((\alpha;\beta), (\gamma; \delta))} \tr(\otimes_{j = 1}^m J_{(a_j + a_j', b_j + b_j')}) \\
  &= u^{-\mathfrak{B}((\alpha;\beta), (\gamma; \delta))} \prod_{j=1}^{m}\tr( J_{(a_j + a_j', b_j + b_j')}) \\
  & = 0.
   \end{align*}
Thus they are pairwise orthogonal.
It remains to show that all $H_\alpha$'s and $H_\infty$ are their own normalizers. We first show that
for $\alpha \neq \alpha' \in \mathbb{F}_{p^m}$ and $\lambda' \in \mathbb{F}_{p^m}^\times$,
\begin{enumerate}[(i)]
\item there is an $\lambda \in \mathbb{F}_{p^m}^\times$ such that $\langle (\lambda; \alpha \lambda), (\lambda'; \alpha' \lambda')\rangle = 1$ and
\item there is an $\lambda \in \mathbb{F}_{p^m}^\times$ such that $\langle (\lambda; \alpha \lambda), (0; \lambda')\rangle = 1$.
\end{enumerate}	
Since the field trace is surjective (see Exercise V.7.2 of \cite{Hung74}), there exists $\gamma \in \mathbb{F}_{p^m}$ such that $Tr _{\mathbb{F}_{p^m}/\mathbb{F}_{p}}(\gamma) = 1$. Thus, we can choose $\lambda = \gamma (\lambda' (\alpha' - \alpha))^{-1}$ for (i) and choose $\lambda = (\lambda')^{-1}$ for (ii). 
 Now, for any $\alpha \in \mathbb{F}_{p^m} $ and $ A \in N_{\mathfrak{sl}_{p^m}(R)}(H_\alpha)$,
 \[
 A = \sum_{\lambda' \in \mathbb{F}_q^\times } \Big(\sum_{\alpha' \in \mathbb{F}_q}a_{(\lambda', \alpha')}\mathcal{J}_{(\lambda', \alpha' \lambda')} + b_{\lambda'} \mathcal{J}_{(0, \lambda')} \Big).
 \]
 For any basis element $\mathcal{J}_{(\lambda, \alpha \lambda)} \in H_\alpha$, we have
  \begin{align}\label{normalbigJ}
 \sum_{\lambda' \in \mathbb{F}_q^\times } \Big(\sum_{\substack{\alpha' \in \mathbb{F}_q \\ \alpha' \neq \alpha}}a_{(\lambda', \alpha')}[\mathcal{J}_{(\lambda', \alpha' \lambda')}, \mathcal{J}_{(\lambda, \alpha \lambda)}] + b_{\lambda'} [\mathcal{J}_{(0, \lambda')}, \mathcal{J}_{(\lambda, \alpha \lambda)} ] \Big) \in H_\alpha.
 \end{align}
 Note that 
 \begin{align*}
 [\mathcal{J}_{(\lambda', \alpha' \lambda')}, \mathcal{J}_{(\lambda, \alpha \lambda)}] &= u^{-\mathfrak{B}((\lambda, \alpha \lambda), (\lambda', \alpha' \lambda'))}(u^{\langle (\lambda', \alpha' \lambda'), (\lambda, \alpha \lambda) \rangle} - 1) \mathcal{J}_{(\lambda'+ \lambda, \alpha' \lambda' + \alpha \lambda)}, \\
 [\mathcal{J}_{(0, \lambda')}, \mathcal{J}_{(\lambda, \alpha \lambda)} ] &= u^{-\mathfrak{B}((\lambda, \alpha \lambda), (0, \lambda'))}(u^{\langle (0, \lambda'), (\lambda, \alpha \lambda) \rangle} - 1) \mathcal{J}_{(\lambda, \lambda'+  \alpha \lambda)}.
 \end{align*}
 The summation in (\ref{normalbigJ})  is also in $\sum_{i \neq \alpha} H_i$.
For any $(\lambda', \alpha')$,  by (i), we can choose a suitable $\lambda$ for which $u^{\langle (\lambda'; \alpha' \lambda'), (\lambda; \alpha \lambda) \rangle} - 1 = u - 1$ is a unit in $R$. This implies  $a_{(\lambda', \alpha')}$ is zero because the sums in (\ref{direct2}) are direct. By (ii), we can show that any $b_{\lambda'}$ is also zero. Thus, $A \in H_\alpha$ and so, $N_{\mathfrak{sl}_{p^m}(R)}(H_k) = H_k$. By using similar arguments, we can show $N_{\mathfrak{sl}_{p^m}(R)}(H_\infty) = H_\infty$.
\end{proof}

We note that Theorem \ref{commu} relies on the existence of a primitive $p$th root of unity $u$ such that $u - 1$ is a unit.
 If $R$ is local, i.e. it has the unique maximal ideal, we can give a sufficient condition for the existence of such a primitive root of unity.

\begin{thm}\label{local}
Let $R$ be a finite local ring with the maximal ideal $M$ and the residue field $k = R/M$. For a prime power $n = p^m$, if $p | |k^\times|$, then $\mathfrak{sl}_n(R)$ has an ODAC 
\[
\mathfrak{sl}_n(R) = H_\infty \oplus H_0 \oplus H_1 \oplus \ldots \oplus H_{n - 1}.
\]
\end{thm}
\begin{proof}
By Theorem XVIII.2 of \cite{M74},
\[
R^\times \cong (1 + M) \times k^\times.
\] 
Thus $p | |R^\times|$ too, so by Cauchy's theorem for finite groups, there exists $u \in R^\times$ of order $p$. Moreover, it follows that $p$ is relatively prime to the characteristic of $R$. Thus, $p \cdot 1$ is a unit in $R$. Next,
we show that $u - 1$ is also a unit in $R$.  Suppose that $u - 1$ is not a unit. Then $u = 1 + x$ for some nonzero $x \in M$. Then $1 = u^p = 1 + px + ( \text {higher power terms of \ } x )$, so $px + (\text{higher power terms of \ } x) = 0$. Let $d > 1$ be the smallest integer such that $x^d = 0$  and multiply the equation by $x^{d - 2}$, we have $px^{d - 1} = 0$, so $x^{d - 1} = 0$ since $p$ is a unit in $R$. A contradiction to the choice of $d$.
\end{proof}

Note that a finite field $\mathbb{F}_q$ is a finite local ring with the maximal ideal $\{ 0 \}$ and $|\mathbb{F}_q^\times| = q-1$, so by the above theorem, we have:

\begin{cor}\label{field}
Let $q$ be a prime power and let $\mathbb{F}_q$ be a finite field of $q$ elements. For another prime power $n = p^m$, if $p | ( q - 1 )$, then $\mathfrak{sl}_n(\mathbb{F}_q)$ has an ODAC
\[
\mathfrak{sl}_n(\mathbb{F}_q) = H_\infty \oplus H_0 \oplus H_1 \oplus \ldots \oplus H_{n - 1}.
\]
\end{cor}

For a finite commutative ring $R$ with $1$,
 $R = R_1 \times R_2 \times \cdots \times R_t$ is a finite direct product of finite local rings (see Theorem VI.2 of \cite{M74}). If each of the local rings in the decomposition of $R$ satisfies the condition in Theorem \ref{local}, then  $\mathfrak{sl}_n(R)$ has an ODAC.

\begin{thm}\label{last}
Let $R = R_1 \times R_2 \times \cdots \times R_t$ be a finite direct product of finite local rings and let $k_i$ be the residue field of $R_i$ for all $i \in \{ 1, 2, \ldots, t\}$. For a prime power $n = p^m$, if $p | |k_i^\times|$ for all $i \in \{1, 2, \ldots, t\}$, then $\mathfrak{sl}_n(R)$ has an ODAC
\[
\mathfrak{sl}_n(R) = H_\infty \oplus H_0 \oplus H_1 \oplus \ldots \oplus H_{n - 1}.
\]
\end{thm}
\begin{proof}
 By the proof of Theorem \ref{local}, there exists a primitive $p$th root of unity $u_i \in R_i^\times$ such that $u_i - 1_{R_i} \in R_i^\times$ for all $i$. Then $u = (u_1, u_2, \ldots, u_t)$ is a primitive $p$th root of unity in $R$ such that 
\begin{align*}
u - 1 &= (u_1, u_2, \ldots, u_t) - (1_{R_1}, 1_{R_2}, \ldots, 1_{R_t}) \\
&= (u_1 - 1_{R_1}, u_2 - 1_{R_2}, \ldots, u_t - 1_{R_t}) \in R^\times
\end{align*}
because $R_1^\times \times R_2^\times \times \ldots \times R_t^\times = R^\times$ (Theorem XVIII.1 of \cite{M74}). Thus Theorem \ref{commu} implies that $\mathfrak{sl}_n(R)$ admits an ODAC.
\end{proof}

By the above theorem, we have the following examples.

\begin{ex} 
 Let $q$ be an odd positive integer and $m$ a positive integer.  Then all prime factors of $q$ are odd and $\mathfrak{sl}_{2^m}(\mathbb{Z}_q)$ has an ODAC.	
\end{ex}	

\begin{ex}
For any  positive integers $s, t$ and $m$,  $\mathfrak{sl}_{3^m}(\mathbb{Z}_{7^s 31^t})$ has an ODAC.
\end{ex}	
 
\section{Concluding remarks}\label{sum}

In Theorems \ref{local}, \ref{last} and Corollary \ref{field}, we provided some sufficient conditions for the existence of an ODAC of $\mathfrak{sl}_n(R)$. These conditions are from the structure of the ring $R$ which can be checked readily. One may ask for what $n$ and $R$, $\mathfrak{sl}_n(R)$ does not have an ODAC. 
We can show the nonexistence of ODAC for a collection of $n$ and $R$. For instance, if $R$ has characteristic $2$, then $\mathfrak{sl}_2(R)$ contains $I_2$. If $A$ is a subalgebra of $\mathfrak{sl}_2(R)$ which is its own normalizer, then $I_2 \in A$. Since
\[
\left[ 
\begin{pmatrix}
0 & 1 \\
0 & 0
\end{pmatrix},
 \begin{pmatrix}
 a & b \\
 c & a
 \end{pmatrix}
 \right],
 \left[ 
 \begin{pmatrix}
 0 & 0 \\
 1 & 0
 \end{pmatrix},
 \begin{pmatrix}
 a & b \\
 c & a
 \end{pmatrix}
 \right] \in A,
\]
we see that $A = \mathfrak{sl}_2(R)$. Therefore, $\mathfrak{sl}_2(R)$ has no proper abelian Cartan subalgebra. Since $\mathfrak{sl}_2(R)$ is not abelian, $\mathfrak{sl}_2(R)$ does not admit an ODAC. In general, if $n $ is a positive multiple of the characteristic of $R$, then the identity matrix $I_n$ is in $ \mathfrak{sl}_n(R)$ and is contained in every abelian Cartan subalgebras. Therefore, each pair of abelian Cartan subalgebras has a non-trivial intersection and thus $\mathfrak{sl}_n(R)$ does not have an ODAC since  $[\mathfrak{sl}_n(R), \mathfrak{sl}_n(R)] = \mathfrak{sl}_n(R)$ implies $\mathfrak{sl}_n(R)$ is nonabelian (i.e. the trivial decomposition is not an ODAC). In particular, if $R = \mathbb{Z}_2, \mathbb{Z}_3, \mathbb{Z}_6$,
  $\mathfrak{sl}_6(R)$ does not possess the desired decomposition. 

We give an example of an algebra that does not have an ODAC, when all the conditions of Theorem \ref{commu} hold except the condition that $u - 1$ being a unit.
Consider $\mathfrak{sl}_3(\mathbb{Z}_9)$. There are two primitive cube roots of unity $4$ and $7$ in $\mathbb{Z}_9$, but $3$ and $6$ are nonunits. Moreover, $3I_3$ is contained in $\mathfrak{sl}_3(\mathbb{Z}_9)$ and also in every abelian Cartan subalgebras. So as in the previous paragraph, we see that $\mathfrak{sl}_3(\mathbb{Z}_9)$ does not have an ODAC.

There is also the problem of uniqueness of ODAC.
It is known that the OD of $\mathfrak{sl}_n(\mathbb{C})$ for all $ n \leq 5$ is unique up to conjugacy \cite{KK83}.  Consider the case $\mathfrak{sl}_2(R)$. If $R = \mathbb{C}$, all Cartan subalgebras are conjugate  \cite{H72}. Thus, we can assume in an OD of $\mathfrak{sl}_2(\mathbb{C})$, one of the Cartan subalgebra consisting of diagonal matrices, so up to conjugacy, an OD looks as follows
\[
\bigg \langle 
\begin{pmatrix}
1 & 0 \\
0 & -1
\end{pmatrix}
\bigg \rangle_\mathbb{C}
\oplus
\bigg \langle 
\begin{pmatrix}
0 & 1 \\
a & 0
\end{pmatrix}
\bigg \rangle_\mathbb{C}
\oplus
\bigg \langle 
\begin{pmatrix}
0 & 1 \\
b & 0
\end{pmatrix}
\bigg \rangle_\mathbb{C}
\] 
for some $a, b \neq 0$. By using the orthogonality with respect to the Killing form, we derive $b = -a$. Note that the conjugation by \[
\begin{pmatrix}
\sqrt{a} & 0 \\
0 & 1
\end{pmatrix}
\]
 stabilizes the first component and maps 
 \[
\begin{pmatrix}
0 & 1 \\
a & 0
\end{pmatrix}
 \longmapsto  \sqrt{a}
\begin{pmatrix}
0 & 1 \\
1 & 0
\end{pmatrix}.
\]
 Therefore, $\mathfrak{sl}_2(\mathbb{C})$ has a unique OD up to conjugacy. For a comparison, consider $R = \mathbb{F}_{p^m}$, where $p \neq 2$. It follows from Lemma \ref{H} and the orthogonality that any classical ODAC of $\mathfrak{sl}_2(\mathbb{F}_{p^m})$ is conjugate to 
\begin{align}\label{formsl2}
\bigg \langle 
\begin{pmatrix}
1 & 0 \\
0 & -1
\end{pmatrix}
\bigg \rangle_{\mathbb{F}_{p^m}}
\oplus
\bigg \langle 
\begin{pmatrix}
0 & 1 \\
a & 0
\end{pmatrix}
\bigg \rangle_{\mathbb{F}_{p^m}}
\oplus
\bigg \langle 
\begin{pmatrix}
0 & 1 \\
-a & 0
\end{pmatrix}
\bigg \rangle_{\mathbb{F}_{p^m}}
\end{align}
for some $a \neq 0$ as well. However, the element $a \in \mathbb{F}_{p^m}$ may not have a square root in $\mathbb{F}_{p^m}$. Consequently, we may not have an automorphism of $\Aut (\mathfrak{sl}_2(\mathbb{F}_{p^m}))$ mapping this decompostion to 
\[
\bigg \langle 
\begin{pmatrix}
1 & 0 \\
0 & -1
\end{pmatrix}
\bigg \rangle_{\mathbb{F}_{p^m}}
\oplus
\bigg \langle 
\begin{pmatrix}
0 & 1 \\
1 & 0
\end{pmatrix}
\bigg \rangle_{\mathbb{F}_{p^m}}
\oplus
\bigg \langle 
\begin{pmatrix}
0 & 1 \\
-1 & 0
\end{pmatrix}
\bigg \rangle_{\mathbb{F}_{p^m}}
\]
as in the complex number case. Similarly, over other fields where not every element has a square root, such as certain finite extensions of $\mathbb{Q}$, the decomposition in (\ref{formsl2}) may not be unique up to conjugation. 

Exploring the possible applications of ODAC over commutative rings requires further attention.
 The OD problem for other algebras has also been studied  \cite{I07}. 
We plan to discuss some of these topics in another paper.

\section*{Acknowledgment}
The authors would like to thank the referee for the constructive comments that lead to the improvement of this paper.
The second author acknowledges the support of a grant from the Simons Foundation ($\# 416937$).

\section{Appendix}
\noindent
Sage code: \\
R.$<$a, b, c, d, x, y, z, u, v, w $>$ = ZZ[] \\
A = matrix([[0, 0, 0], [a, 0, b], [c, d, 0]]) \\
B = matrix([[0, x, y], [u, 0, z], [v, w, 0]]) \\
C = A*B - B*A \\
detkilling = (A*A).trace()*(B*B).trace() - ((A*B).trace())\^ \ 2 \\
J = ideal (list (C[0]) + list (C[1]) + list (C[2]) + [a*b*c*d]) \\
detkilling in J \\

\noindent
Magma code: \\
 P$<$a,b,c,d,x, y,z,u,v,w$>$ := PolynomialRing(IntegerRing(),10); \\
A := Matrix(3, [0,0,0, a,0,b, c,d,0]);\\
B := Matrix(3, [0,x,y, u,0,z, v,w,0]);\\
C := A*B - B*A ;\\
detkilling := Trace(A*A)*Trace(B*B) - Trace(A*B)\^ \ 2;\\
S :=  \{ C[i,j]: i,j in [1, 2, 3] \} join \{ a*b*c*d \};\\
J := Ideal(S);\\
detKilling in J;\\
\end{document}